\documentclass[reqno]{amsart}

\usepackage[english]{babel}
\usepackage{amssymb,amsmath,amsthm,hyperref,cleveref,comment,bm,fontenc}

\newtheorem{thrm}{Theorem}[section]
\newtheorem{cor}[thrm]{Corollary}
\newtheorem{lem}[thrm]{Lemma}
\newtheorem{prop}[thrm]{Proposition}

\theoremstyle{definition}
\newtheorem{defn}[thrm]{Definition}
\newtheorem{exm}[thrm]{Example}

\newtheorem{rem}[thrm]{Remark}

\crefrangeformat{equation}{#3(#1)#4--#5(#2)#6}

\crefname{thrm}{Theorem}{Theorems}
\crefname{lem}{Lemma}{Lemmas}
\crefname{cor}{Corollary}{Corollaries}
\crefname{prop}{Proposition}{Propositions}
\crefname{defn}{Definition}{Definitions}
\crefname{exm}{Example}{Examples}
\crefname{rem}{Remark}{Remarks}
\crefname{section}{Section}{Sections}
\crefname{equation}{\unskip}{\unskip}
\crefname{enumi}{\unskip}{\unskip}

\newcommand{\pA}{{\sf pA}}
\newcommand{\A}{{\sf A}}
\newcommand{\R}[1]{{\sf R}{(#1)}}

\newcommand{\sem}{{\sf Sem}}
\newcommand{\grp}{{\sf Grp}}
\newcommand{\alg}{{\sf Alg}}
\newcommand{\V}{{\sf V}}

\newcommand{\pAl}[1]%
{%
 {\sf pAl}{(#1)}%
}

\newcommand{\Al}[1]%
{%
 {\sf Al}{(#1)}%
}

\newcommand{\mor}[2]%
{%
 \mathrm{Mor}(#1,#2)%
}

\newcommand{\id}{\mathrm{id}}

\newcommand{\cD}{\mathcal D}
\newcommand{\cA}{\mathcal A}
\newcommand{\cB}{\mathcal B}
\newcommand{\cR}{\mathcal R}
\newcommand{\cW}{\mathcal W}
\newcommand{\cS}{\mathcal S}
\newcommand{\cT}{\mathcal T}
\newcommand{\cH}{\mathcal H}

\newcommand{\U}{\mathbf U}
\newcommand{\bt}{\bm\theta}
\newcommand{\bth}{\bm\theta}

\newcommand{\bA}{\mathbf A}

\newcommand{\dom}[1]%
{%
 \operatorname{\mathrm{dom}}{#1}%
}

\newcommand{\ran}[1]%
{%
 \operatorname{\mathrm{ran}}{#1}%
}

\newcommand{\impl}{\Rightarrow}

\newcommand{\m}{^{-1}}

\newcommand{\s}{\sigma}
\newcommand{\0}{\theta}
\newcommand{\vt}{\vartheta}

\newcommand{\g}{\gamma}
\newcommand{\vf}{\varphi}
\newcommand{\ve}{\varepsilon}

\newcommand{\Th}{\Theta}
\newcommand{\vTh}{\varTheta}
\newcommand{\Sg}{\Sigma}
\newcommand{\vS}{\varSigma}

\newcommand{\emp}{\varnothing}
\renewcommand{\iff}{\Leftrightarrow}
\newcommand{\nat}{\natural}
\newcommand{\ol}{\overline}
\newcommand{\wt}{\widetilde}
\newcommand{\la}{\langle}
\newcommand{\ra}{\rangle}

\begin{document}

\title{Reflectors and globalizations of partial actions of groups}

\author[Mykola Khrypchenko]{Mykola Khrypchenko\textsuperscript{*}}
\address{Departamento de Matem\'atica, Universidade Federal de Santa Catarina, Campus Reitor Jo\~ao David Ferreira Lima, Florian\'opolis, SC, CEP: 88040--900, Brazil}
\email{nskhripchenko@gmail.com}

\thanks{\textsuperscript{*}Partially supported by FAPESP of Brazil (process: 2012/01554--7)}

\author[Boris Novikov]{Boris Novikov\textsuperscript{\textdagger}}
\address{Department of Mechanics and Mathematics, V. N. Karazin Kharkiv National University, Svobody sq. 4, Kharkiv, 61077, Ukraine}
\email{boris.v.novikov@univer.kharkov.ua}

\thanks{\textsuperscript{\textdagger}Deceased 30 March 2014}


\begin{abstract}
 Given a partial action $\0$ of a group on a set with an algebraic structure, we construct a reflector of $\0$ in the corresponding subcategory of global actions and study the question when this reflector is a globalization. In particular, if $\0$ is a partial action on an algebra from a variety $\V$, then we show that the problem reduces to the embeddability of certain generalized amalgam of $\V$-algebras associated with $\0$. As an application, we describe globalizable partial actions on semigroups, whose domains are ideals.
\end{abstract}

\maketitle

\section*{Introduction}\label{intro}

Recall from \cite[Definition 1.2]{E1} that a partial action of a group $G$ on a set $A$ is a collection of bijections $\0=\{\0_x:D_{x\m}\to D_x\}_{x\in G}$, where $D_x\subseteq A$ and
\begin{enumerate}
 \item $D_1=A$ with $\0_1=\id_A$;
 \item $\0_x(D_{x\m}\cap D_y)=D_x\cap  D_{xy}$;
 \item $\0_x\circ\0_y=\0_{xy}$ on $D_{y\m}\cap D_{y\m x\m}$.
\end{enumerate}
The subsets $D_x$ are called the {\it domains} of $\0$. When $A$ possesses an extra structure, one naturally assumes that $\0$ respects this structure in some sense. For example, in the algebra~\cite{DE}, ring~\cite{DRS} or semigroup~\cite{DN} setting the domains $D_x$ are supposed to be ideals of $A$, and the partial bijections $\0_x$ are isomorphisms. An alternative (equivalent) definition of a partial action, in terms of a partially defined map $\0$ from $G\times A$ to $A$, can be found in~\cite{KL} (see \cref{part-act-defn}).

A natural way to construct a partial action $\0$ of a group $G$ on a set $A$ is to restrict a global action of $G$ on some bigger set to $A$. It was proved by F.~Abadie in~\cite{Abadie03}, in the context of partial actions on topological spaces, and independently by J.~Kellendonk and M.~V.~Lawson in~\cite{KL} that, up to an isomorphism, each partial action $\0$ of $G$ on $A$ can be obtained this way. The corresponding global action is called an enveloping action of $\0$ in~\cite{Abadie03} and a globalization of $\0$ in~\cite{KL}. We shall follow the terminology of~\cite{KL}. Both~\cite{Abadie03} and~\cite{KL} give the same explicit construction of a globalization of $\0$, which we denote by $\0^U$. Moreover, it is proved in~\cite[Theorem 3.4]{KL} that $\0^U$ is universal among all the globalizations of $\0$, while~\cite[Theorem 1.1]{Abadie03} contains a stronger result saying that $\0^U$ is a reflector of $\0$ in the subcategory of global actions.

For a partial action of a group on an associative algebra $\cA$ the question of existence of a globalization was first considered in~\cite{DE}. If $\cA$ is unital, then it was proved in~\cite[Theorem 4.5]{DE} that $\0$ admits a globalization if and only if each ideal $D_x$ is unital. This criterion was generalized to so-called left $s$-unital rings in~\cite[Theorem 3.1]{DRS}. It was also used in~\cite{Ferrero06} to prove that a partial action on a semiprime ring is globalizable.

In this paper we study the globalization problem in the universal algebra setting. The idea is to construct a reflector of a partial action in the corresponding subcategory of global actions. Then the problem reduces to the verification whether the reflector is a globalization, and whenever it is, such a globalization is automatically universal.

The article is organized as follows. In \cref{pA} we fix the notations and recall well-known facts about partial actions of groups on sets and their globalizations. 

In \cref{pA(GR(T))} we introduce the notion of a partial action of a group on a relational system and show that it admits a universal globalization which is a reflector (see \cref{rho^U-glob}). 

We proceed to partial actions on (in general, partial) algebras in \cref{pA(GpAl(T))}. The domains of such partial actions are assumed to be relative subalgebras. \cref{ex-glob-pA(GpAl(T))<=>R_A^U-funct} gives a necessary and sufficient condition for the existence of a globalization of a partial action on a partial algebra, which holds, in particular, when the domains are subalgebras (see \cref{D_x-subalg}). In the case of partial actions on total algebras we begin with the construction of a reflector in the subcategory of global actions, which uses the concept of the algebra absolutely freely generated by a partial algebra (see \cref{refl-A(GAl(T))-pA(GAl(T)),refl-A(GAl(T))-pA(GAl(T))-precise}), and then show that the reflector is a globalization in \cref{glob-A(GAl(T))}. Restricting ourselves to partial actions on algebras from some fixed variety, we can still construct a reflector (see \cref{A(GV)-refl-in-pA(GV),refl-A(GV),refl-A(GV)-equiv-form}), which may not be 
a globalization in general. A characterization of globalizable partial actions is given in \cref{glob-pA(GV)}.

It turns out that the globalization problem is closely related to the embeddability problem for generalized amalgams, which was investigated, in the group case, by B.~H.~Neumann and H.~Neumann in~\cite{HNeumann48,HNeumann49,HNeumann50,HNeumann51,BNeumann-HNeumann53} (see also the survey~\cite{BNeumann54}). More precisely, with any partial action $\0$ of a group on an algebra from a variety $\V$ we associate a generalized amalgam $\bA$ of $\V$-algebras, such that $\0$ is a globalizable if and only if $\bA$ is embeddable (see \cref{refl-glob<=>amalg-emb}). 

In \cref{D_x-ideal} we use \cref{glob-pA(GV)} and the technique of reduction systems to obtain a criterion of the existence of a globalization of a partial action $\0$ on a semigroup, whose domains are ideals (see \cref{glob-A(GSem)-D_x-ideal}). In particular, when the domains of $\0$ are idempotent or weakly-reductive semigroups, the partial action is globalizable (see \cref{D_x-idemp-or-weakly-red}). If, moreover, the domains of $\0$ are unital ideals, then we show in \cref{D_x-unital} that $\0$ is isomorphic to the restriction of a global action to some ideal.

%
\section{Preliminaries}\label{pA}

Let $A$ and $B$ be sets. We use the notation $f:A\dashrightarrow B$ for a {\it partial map} from $A$ to $B$. If $f(a)$ is not defined, then we write $f(a)=\emp$.

It will be convenient to us to use the following equivalent definition of a partial action, which is due to~\cite{KL}.
\begin{defn}\label{part-act-defn}
Let $G$ be a group and $A$ a set. A {\it partial action of $G$ on $A$} is a partial map $\0:G\times A \dashrightarrow A$ satisfying 
\begin{enumerate}
 \item $\emp\ne\0(1,a)=a$;\label{1a=a}
 \item $\0(x,a)\ne\emp\ \impl\ \emp\ne\0(x\m,\0(x,a))=a$;\label{x-inv-xa=a}
 \item $\0(y,a)\ne\emp\ \&\ \0(x,\0(y,a))\ne\emp\ \impl\ \emp\ne\0(xy,a)=\0(x,\0(y,a))$.\label{(xy)a=x(ya)}
\end{enumerate}
If $\0$ is everywhere defined, then we say that $\0$ is {\it global}.
\end{defn}
This definition is equivalent to~\cite[Definition 1.2]{E1} mentioned in the introduction with $D_x=\ran{\0(x,-)}$ and $\0_x=\0(x,-)$.

A set $A$ with a partial action $\theta$ on it will be denoted by $(\0,A)$. The pairs $(\0,A)$ form a category $\pA(G)$, in which a morphism from $(\0,A)$ to $(\0',A')$ is a map $\vf:A\to A'$, such that
\begin{align}\label{morph-in-pA}
 \0(x,a)\ne\emp\ \impl\ \emp\ne\0'(x,\vf(a))=\vf(\0(x,a)).
\end{align}
The full subcategory consisting of the pairs $(\0,A)$ with global $\0$ will be denoted by $\A(G)$.

In what follows, when there is no confusion, we shall write $xa$ for $\0(x,a)$. For example, \cref{morph-in-pA} can be rewritten as
\begin{align}\label{morph-in-pA-new}
 xa\ne\emp\ \impl\ \emp\ne x\vf(a)=\vf(xa).
\end{align}

Partial actions appear in the following situation. Suppose that we are given a global action $\vt$ of $G$ on $B$ and $A\subseteq B$. Denote by $\0$ the {\it restriction} of $\vt$ to $A$, i.\,e. the partial map $G\times A\dashrightarrow A$, such that $\0(x,a)=b\iff\vt(x,a)=b\in A$ for any $a\in A$. Then $\0$ is a partial action of $G$ on $A$.

The converse construction is called a globalization.
\begin{defn}\label{glob-in-pA-defn}
A {\it globalization}~\cite{Abadie03,KL} of a pair $(\0,A)\in\pA(G)$ is a pair $(\vt,B)\in\A(G)$ with an injective morphism $\iota:(\0,A)\to(\vt,B)$ such that for all $x\in G$ and $a\in A$
\begin{align}\label{xa-ne-emp-iff-xiota(a)-in-iota(M)}
 xa\ne\emp\ \iff\ x\iota(a)\in\iota(A).
\end{align}
\end{defn}
Observe that it is sufficient to require the ``only if'' part of \cref{xa-ne-emp-iff-xiota(a)-in-iota(M)}, since the ``if'' part follows from \cref{morph-in-pA-new}. The morphism $\iota$ itself will also be called a globalization.

\begin{defn}\label{univ-glob-in-pA-defn}
A globalization $\iota:(\0,A)\to(\vt,B)$ is called {\it universal}, if for any globalization  $\iota':(\0,A)\to(\vt',B')$ there exists a unique morphism $\kappa:(\vt,B)\to(\vt',B')$ such that $\iota'=\kappa\circ\iota$.
\end{defn}

For any pair $(\0,A)$ there exists a (unique up to an isomorphism) universal globalization $(\0^U,A^U)$, which can be constructed as follows. Define the equivalence relation $\sim$ on $G\times A$:
$$
 (x,a)\sim(y,b)\ \iff\ \emp\ne(y\m x)a=b
$$
and set $A^U=(G\times A)/{\sim}$. Let $[x,a]$ be the $\sim$-class containing $(x,a)$. The action $\0^U$ of $G$ on $A^U$ has the form $x[y,a]=[xy,a]$, and the map $[1,-]$ sending $a\in A$ to $[1,a]\in A^U$ is the desired injection. Moreover, note that for any globalization $\iota:(\0,A)\to(\vt,B)$ the corresponding morphism $\kappa:(\0^U,A^U)\to(\vt,B)$ is injective (see~\cite[Theorem 1.1]{Abadie03} and~\cite[Theorem 3.4]{KL}).


Recall the definition of a reflector (see~\cite[IV.3]{Maclane-Cat} or~\cite[IV.4]{Tsalenko-Shulgeyfer}).

\begin{defn}\label{refl-defn}
A subcategory ${\sf D}$ of a category ${\sf C}$ is called {\it reflective}, if for any object $C\in{\sf C}$ there exists an object $R_{\sf D}(C)\in{\sf D}$ (called a ${\sf D}$-{\it reflector} of $C$) and a morphism $\ve_{\sf D}(C): C\to R_{\sf D}(C)$ (a {\it reflector morphism}), such that for all $D\in{\sf D}$ and $\vf:C\to D$ there is a unique $\psi:R_{\sf D}(C)\to D$ in ${\sf D}$ with $\psi\circ\ve_{\sf D}(C)=\vf$.
\end{defn}

\begin{rem}\label{refl-functor}
The map $C\mapsto R_{\sf D}(C)$ is a functor ${\sf C}\to{\sf D}$, which is left adjoint to the inclusion functor ${\sf D}\to{\sf C}$. In particular, any two $\sf D$-reflectors of $C$ are naturally isomorphic.
\end{rem}

The following fact is due to F.~Abadie (see~\cite[Theorem 1.1]{Abadie03}).
\begin{prop}\label{(0^UA^U)-refl}
The subcategory $\A(G)$ is reflective in $\pA(G)$. More precisely, $(\0^U,A^U)$ is an $\A(G)$-reflector of $(\0,A)\in\pA(G)$.
\end{prop}
%

\section{Partial actions on relational systems}\label{pA(GR(T))}

The basic notions of model theory that we use in this section can be found in~\cite[\S 36]{Gratzer}.

A {\it type of relational systems} is a sequence $T=\{n_\g\}_{\g<o(T)}$ of positive integers indexed by ordinal numbers. For each $\g<o(T)$ we fix a symbol $\rho_\g$ of an $n_\g$-ary relation. A {\it relational system of type $T$} is a pair $\cA=(A,R)$, where $A$ is a set and $R=\{(\rho_\g)_\cA\}_{\g<o(T)}$, each $(\rho_\g)_\cA$ being an $n_\g$-ary relation on $A$. We shall write $\rho_\g$ for $(\rho_\g)_\cA$, if this does not lead to confusion. 

Relational systems of the same type $T$ form a category $\R T$. A {\it morphism} $\vf:\cA\to\cA'$ is a map $\vf:A\to A'$, such that $\vf(\rho_\g)\subseteq\rho_\g$, i.\,e.
$$
 (a_1,\dots,a_{n_\g})\in\rho_\g\ \impl\ (\vf(a_1),\dots,\vf(a_{n_\g}))\in\rho_\g.
$$
A system $\cA$ is a {\it subsystem} of $\cB$, if $A\subseteq B$ and $(\rho_\g)_\cA=(\rho_\g)_\cB\cap A^{n_\g}$.

Let $G$ be a group and $\cA$ a relational system. By an {\it action of $G$ on $\cA$} we mean an action of $G$ on $A$, which preserves each $\rho_\g\in R$ in the sense that $x\rho_\g\subseteq\rho_\g$, i.\,e.
$$
 (a_1,\dots,a_{n_\g})\in\rho_\g\impl(xa_1,\dots,xa_{n_\g})\in\rho_\g.
$$

\begin{rem}\label{restr-to-(MR)}
Let $\cB$ be a relational system and $\cA$ a subsystem of $\cB$. Then the restriction of an action of $G$ on $\cB$ to $A$ satisfies
\begin{align}\label{inv-rho_g}
 (a_1,\dots,a_{n_\g})\in\rho_\g\ \&\ xa_1,\dots,xa_{n_\g}\ne\emp\ \impl\ (xa_1,\dots,xa_{n_\g})\in\rho_\g.
\end{align}
\end{rem}

This motivates the following definition.
\begin{defn}\label{pa-on-(MR)-defn}
Let $G$ be a group and $\cA$ a relational system. A {\it partial action of $G$ on $\cA$} is a partial action of $G$ on $A$ satisfying \cref{inv-rho_g}.
\end{defn}

Denote by $\pA(G,\R T)$ the category whose objects are the pairs $(\0,\cA)$, where $\cA\in\R T$ and $\0$ is a partial action of $G$ on $\cA$. A morphism from $(\0,\cA)$ to $(\0',\cA')$ is a morphism $\cA\to\cA'$ of relational systems, which is at the same time a morphism of partial actions $(\0,A)\to(\0',A')$. The pairs $(\0,\cA)$ with global $\0$ form a subcategory $\A(G,\R T)$.


\begin{defn}\label{glob-pA(GR(T))-defn}
A {\it globalization of} $(\0,\cA)\in\pA(G,\R T)$ is $(\vt,\cB)\in\A(G,\R T)$ with an injective morphism $\iota:(\0,\cA)\to(\vt,\cB)$ such that
\begin{enumerate}
 \item $\iota:(\0,A)\to(\vt,B)$ is a globalization in $\pA(G)$,\label{iota-glob-in-pA}
 \item $(\iota(a_1),\dots,\iota(a_{n_\g}))\in\rho_\g\ \impl\ (a_1,\dots,a_{n_\g})\in\rho_\g$.\label{iota(A)-subsystem}
\end{enumerate}
\end{defn}

Observe that condition \cref{iota(A)-subsystem} in \cref{glob-pA(GR(T))-defn} says that $\iota(\cA)=(\iota(A),\{\iota(\rho_\g)\}_\g)$ is a subsystem of $\cB$.

\begin{thrm}\label{rho^U-glob}
Let $(\0,\cA)\in\pA(G,\R T)$. Define $\cA^U$ to be the relational system on $A^U$ with
$$
(\rho_\g)_{\cA^U}=\{([x,a_1],\dots,[x,a_{n_\g}])\mid(a_1,\dots,a_{n_\g})\in(\rho_\g)_\cA,\ x\in G\}.
$$
Then $(\0^U,\cA^U)$ is a universal globalization of $(\0,\cA)$ in $\A(G, \R T)$. Moreover, $(\0^U,\cA^U)$ is an $\A(G, \R T)$-reflector of $(\0,\cA)$.
\end{thrm}
\begin{proof}
To prove that $(\0^U,\cA^U)$ is a globalization, it is sufficient to check \cref{iota(A)-subsystem} of \cref{glob-pA(GR(T))-defn}. Let $([1,a_1],\dots,[1,a_{n_\g}])\in\rho_\g$. Then there are $x\in G$ and $(a'_1,\dots,a'_{n_\g})\in\rho_\g$, such that $[1,a_i]=[x,a'_i]$, and hence $\emp\ne xa'_i=a_i$ for all $i$. Therefore, $(a_1,\dots,a_{n_\g})\in\rho_\g$ by \cref{inv-rho_g}.

The universality of $(\0^U,\cA^U)$ will follow from the reflectivity. Take $(\vt,\cB)\in\A(G,\R T)$ and $\vf:(\0,\cA)\to(\vt,\cB)$. According to \cref{(0^UA^U)-refl} there is a unique morphism $\psi:(\0^U,A^U)\to(\vt,B)$ satisfying $\vf=\psi\circ[1,-]$, and we only need to show that $\psi(\rho_\g)\subseteq\rho_\g$. Indeed, if $([x,a_1],\dots,[x,a_{n_\g}])\in\rho_\g$, then
$$
 \psi([x,a_1],\dots,[x,a_{n_\g}])=(x\vf(a_1),\dots,x\vf(a_{n_\g}))\in x\vf(\rho_\g)\subseteq x\rho_\g\subseteq\rho_\g.
$$
\end{proof}

\section{Partial actions on algebras}\label{pA(GpAl(T))}

The terminology that we use in this section is due to Gr\"{a}tzer~\cite{Gratzer}.

Let $A$ be a set and $n\ge 0$. A {\it partial $n$-ary operation on $A$} is a partial function $f: A^n\dashrightarrow A$, where $A^0$ is by definition a singleton $\{\emptyset\}$. If $n=0$, then $f$ will be identified with $f(\emptyset)$ (which may be undefined). We shall frequently use the expressions of the form $f(a_1,\dots,a_n)$, which for $n=0$ should be understood as $f$.

A {\it type of algebras} is a sequence $T=\{n_\g\}_{\g<o(T)}$ of non-negative integers indexed by ordinal numbers. For each $\g<o(T)$ we fix a symbol $f_\g$ of an $n_\g$-ary operation. A {\it partial algebra of type $T$} is a pair $\cA=(A,F)$, where $A$ is a set and $F=\{(f_\g)_\cA\}_{\g<o(T)}$, each $(f_\g)_\cA$ being a partial $n_\g$-ary operation on $A$. We shall often omit the index $\cA$, writing $f_\g$ for $(f_\g)_\cA$, if $\cA$ is clear from the context. If $o(T)=1$, then we write $f$ for $f_0$ and $(A,f)$ for $(A,\{f\})$.

Let $\cA$ and $\cB$ be partial algebras of the same type. A map $\vf:A\to B$ is a {\it homomorphism} $\cA\to\cB$, if
$$
 f_\g(a_1,\dots,a_{n_\g})\ne\emp\impl\emp\ne f_\g(\vf(a_1),\dots,\vf(a_{n_\g}))=\vf(f_\g(a_1,\dots,a_{n_\g})).
$$
An {\it isomorphism} is a bijective homomorphism whose inverse is also a homomorphism.

An equivalence relation $\Th$ on $A$ is a {\it congruence} on $\cA$, if for all $(a_i,b_i)\in\Th$, $1\le i\le n_\g$, one has
\begin{align}\label{subst-prop}
 f_\g(a_1,\dots,a_{n_\g}),f_\g(b_1,\dots,b_{n_\g})\ne\emp\impl(f_\g(a_1,\dots,a_{n_\g}),f_\g(b_1,\dots,b_{n_\g}))\in\Th.
\end{align}
The $\Th$-class of $a\in A$ will be denoted by $[a]_\Th$. The quotient set $A/\Th$ has the natural structure of a partial algebra, which we denote by $\cA/\Th$. The map $a\mapsto[a]_\Th$ is an epimorphism $\cA\to\cA/\Th$, for which we use the notation $\Th^\nat$. Sometimes it is convenient to write $(a,b)\in\Th$ as $a\sim_\Th b$. Given a homomorphism $\vf:\cA\to\cB$, the {\it kernel} of $\vf$ is the congruence $\ker\vf=\{(a,b)\in A^2\mid\vf(a)=\vf(b)\}$ on $\cA$.

For partial algebras there are several notions of a subalgebra. A partial algebra $\cB$ is a {\it relative subalgebra} of $\cA$, if $B\subseteq A$ and the partial operations on $\cB$ are the restrictions of the partial operations on $\cA$ to $B$ in the sense that
$$
(f_\g)_\cB(b_1,\dots,b_{n_\g})\ne\emp\iff\emp\ne (f_\g)_\cA(b_1,\dots,b_{n_\g})\in B.
$$
If $B$ is closed under all $f_\g$, i.\,e. 
$$
(f_\g)_\cA(b_1,\dots,b_{n_\g})\ne\emp\impl (f_\g)_\cA(b_1,\dots,b_{n_\g})\in B,
$$
then a relative subalgebra $\cB$ is called a {\it subalgebra} of $\cA$. 

Observe that each partial operation $f:A^n\dashrightarrow A$ can obviously be identified with an $(n+1)$-ary relation
$$
 \rho=\{(a_1,\dots,a_n,f(a_1,\dots,a_n))\mid  f(a_1,\dots,a_n)\ne\emp\}\subseteq A^{n+1}.
$$
Conversely, if $\rho\subseteq A^{n+1}$ satisfies
$$
(a_1,\dots,a_n,b),(a_1,\dots,a_n,c)\in\rho\,\impl\,b=c,
$$
then $\rho$ comes from some (uniquely defined) $f:A^n\dashrightarrow A$. When $n=1$, this property of $\rho$ is called {\it functionality}. We shall use the same terminology for $n\ne 1$ as well. Moreover, we shall say that a relational system is {\it functional}, if all its relations are functional.

Denote by $\pAl T$ the category of partial algebras of type $T$ and homomorphisms between them. As it was observed above, each $\cA\in\pAl T$ can be seen as a functional relational system $\cR(\cA)$ of type $T+1$ on $A$, where $T+1:=\{n_\g+1\}_{\g<o(T)}$. This defines an embedding functor of $\pAl T$ into $\R{T+1}$.

\begin{defn}\label{pApO-defn}
 Given a group $G$ and $\cA\in\pAl T$, a {\it partial action of $G$ on $\cA$} is by definition a partial action of $G$ on $\cR(\cA)$.
\end{defn}

Taking into account \cref{inv-rho_g}, we see that a partial action of $G$ on $\cA$ is a partial action of $G$ on $A$ satisfying
\begin{align}\label{pA(GpAl(T))-cond}
 xa_1,\dots,xa_{n_\g},xf_\g(a_1,\dots,a_{n_\g})\ne\emp\impl\emp\ne f_\g(xa_1,\dots,xa_{n_\g})=xf_\g(a_1,\dots,a_{n_\g}).
\end{align}
Here $xf_\g(a_1,\dots,a_{n_\g})\ne\emp$ assumes that $f_\g(a_1,\dots,a_{n_\g})\ne\emp$. If $n_\g=0$, then \cref{pA(GpAl(T))-cond} should be understood as $xf_\g\ne\emp\impl xf_\g=f_\g$. 

Note that \cref{pA(GpAl(T))-cond} is equivalent to the fact that $\0_x$ is an isomorphism $\cD_{x\m}\to\cD_x$, where $\cD_x$ is the relative subalgebra on $D_x$.

The full subcategory of $\pA(G,\R{T+1})$ consisting of partial actions of $G$ on partial algebras of type $T$ will be denoted by $\pA(G,\pAl T)$. The corresponding subcategory of global actions is $\A(G,\pAl T)=\pA(G,\pAl T)\cap\A(G,\R{T+1})$. 
\begin{defn}\label{glob-A(GpAl(T))}
 By a {\it globalization} of $(\0,\cA)\in\pA(G,\pAl T)$ in $\A(G,\pAl T)$ we mean $(\vt,\cB)\in\A(G,\pAl T)$, such that $(\vt,\cR(\cB))$ is a globalization of $(\0,\cR(\cA))$ in $\A(G,\R{T+1})$. 
\end{defn}
 Observe that an injective morphism $\iota:(\0,\cA)\to(\vt,\cB)$ is a globalization if and only if $\iota(\cA)$ is a relative subalgebra of $\cB$ and \cref{xa-ne-emp-iff-xiota(a)-in-iota(M)} holds.

\begin{thrm}\label{ex-glob-pA(GpAl(T))<=>R_A^U-funct}
 There exists a globalization of $(\0,\cA)\in\pA(G,\pAl T)$ in $\A(G,\pAl T)$ if and only if $\cR(\cA)^U$ is functional. In this case $(\0^U,\cA^U)$ is a universal globalization of $(\0,\cA)$ in $\A(G,\pAl T)$, where $\cA^U$ is the partial algebra structure on $A^U$, such that $\cR(\cA^U)=\cR(\cA)^U$, i.\,e.
 $$
  (\rho_\g)_{\cA^U}([x,a_1],\dots,[x,a_{n_\g}])=[x,\rho_\g(a_1,\dots,a_{n_\g})].
 $$
 Moreover, $(\0^U,\cA^U)$ is an $\A(G,\pAl T)$-reflector of $(\0,\cA)$.
\end{thrm}
\begin{proof}
 The ``if'' part is obvious. For the ``only if'' part suppose that $\iota:(\0,\cA)\to(\vt,\cB)$ is a globalization of $(\0,\cA)$ in $\A(G,\pAl T)$. Then by \cref{rho^U-glob} there is an injective $\kappa:(\0^U,\cR(\cA)^U)\to(\vt,\cR(\cB))$ in $\A(G,\R{T+1})$. Since $\kappa(\rho_\g)\subseteq\rho_\g$ and $\cR(\cB)$ is functional, it follows that $\cR(\cA)^U$ is also functional. The second affirmation of the theorem is immediate.
\end{proof}
 
 The following proposition clarifies \cref{ex-glob-pA(GpAl(T))<=>R_A^U-funct}.
\begin{prop}\label{R_A^U-funct-crit}
 Let $(\0,\cA)\in\pA(G,\pAl T)$. Then $(\0,\cA)$ is globalizable if and only if\footnote{The equality in \cref{glob-pA(GpAl(T))-cond} should be understood in a partial sense. If $n_\g=0$, then \cref{glob-pA(GpAl(T))-cond} takes the form $f_\g\ne\emp\impl f_\g=xf_\g$.}
 \begin{align}\label{glob-pA(GpAl(T))-cond}
  f_\g(a_1,\dots,a_{n_\g}),xa_1,\dots,xa_{n_\g}\ne\emp\impl f_\g(xa_1,\dots,xa_{n_\g})=xf_\g(a_1,\dots,a_{n_\g}).
 \end{align}
\end{prop}
\begin{proof}
 We show that \cref{glob-pA(GpAl(T))-cond} is equivalent to the functionality of $\cR(\cA)^U$. Assume that $\cR(\cA)^U$ is functional and take $a_1,\dots,a_{n_\g}\in A$ and $x\in G$ for which the left-hand side of \cref{glob-pA(GpAl(T))-cond} holds.
 If $xf_\g(a_1,\dots,a_{n_\g})\ne\emp$, then $\emp\ne f_\g(xa_1,\dots,xa_{n_\g})=xf_\g(a_1,\dots,a_{n_\g})$ by \cref{pA(GpAl(T))-cond}. Now suppose that $f_\g(xa_1,\dots,xa_{n_\g})\ne\emp$. Observe first that $f_\g(a_1,\dots,a_{n_\g})\ne\emp$ implies
 \begin{align*}
  (a_1,\dots,a_{n_\g},f_\g(a_1,\dots,a_{n_\g}))\in\rho_\g &\impl([1,a_1],\dots,[1,a_{n_\g}],[1,f_\g(a_1,\dots,a_{n_\g})])\in\rho_\g\\
  &\impl([x,a_1],\dots,[x,a_{n_\g}],[x,f_\g(a_1,\dots,a_{n_\g})])\in\rho_\g.
 \end{align*}
 Similarly $xa_1,\dots,xa_{n_\g},f_\g(xa_1,\dots,xa_{n_\g})\ne\emp$ yields
 $$
  ([1,xa_1],\dots,[1,xa_{n_\g}],[1,f_\g(xa_1,\dots,xa_{n_\g})])\in\rho_\g.
 $$
 But $[x,a_1]=[1,xa_1],\dots,[x,a_{n_\g}]=[1,xa_{n_\g}]$, hence 
 $$
 [x,f_\g(a_1,\dots,a_{n_\g})]=[1,f_\g(xa_1,\dots,xa_{n_\g})].
 $$
 The latter means that $\emp\ne xf_\g(a_1,\dots,a_{n_\g})=f_\g(xa_1,\dots,xa_{n_\g})$.
 
 Conversely, assume \cref{glob-pA(GpAl(T))-cond} and let
 \begin{align*}
  &([x,a_1],\dots,[x,a_{n_\g}],[x,f_\g(a_1,\dots,a_{n_\g})])\in\rho_\g,\\
  &([y,b_1],\dots,[y,b_{n_\g}],[y,f_\g(b_1,\dots,b_{n_\g})])\in\rho_\g
 \end{align*}
 with $[x,a_1]=[y,b_1],\dots,[x,a_{n_\g}]=[y,b_{n_\g}]$. We have 
 $$
 f_\g(a_1,\dots,a_{n_\g})\ne\emp,\ \ \emp\ne(y\m x)a_1=b_1,\dots,\emp\ne(y\m x)a_{n_\g}=b_{n_\g}
 $$
 and $f_\g((y\m x)a_1,\dots,(y\m x)a_{n_\g})=f_\g(b_1,\dots,b_{n_\g})\ne\emp$. So $(y\m x)f_\g(a_1,\dots,a_{n_\g})$ is defined  and thus equals $f_\g(b_1,\dots,b_{n_\g})$ by \cref{glob-pA(GpAl(T))-cond}. It follows that
 $$
 [x,f_\g(a_1,\dots,a_{n_\g})]=[y,f_\g(b_1,\dots,b_{n_\g})].
 $$
\end{proof}

\begin{rem}\label{D_x-subalg}
 Let $(\0,\cA)\in\pA(G,\pAl T)$, such that $\cD_x$ is a subalgebra of $\cA$ for all $x$. Then there is a globalization of $(\0,\cA)$ in $\A(G,\pAl T)$.
\end{rem}
\noindent Indeed, suppose that $f_\g(a_1,\dots,a_{n_\g}),xa_1,\dots,xa_{n_\g}\ne\emp$. The latter means that $a_1,\dots,a_{n_\g}\in D_{x\m}$ and $f_\g(a_1,\dots,a_{n_\g})\ne\emp$. Since $\cD_{x\m}$ is a subalgebra, then $f_\g(a_1,\dots,a_{n_\g})\in D_{x\m}$, i.\,e. $xf_\g(a_1,\dots,a_{n_\g})\ne\emp$, and hence \cref{pA(GpAl(T))-cond} yields $\emp\ne f_\g(xa_1,\dots,xa_{n_\g})=xf_\g(a_1,\dots,a_{n_\g})$. \\

The next example shows that the condition that $\cD_x$ is a subalgebra is not necessary in general for the existence of a globalization.

\begin{exm}\label{D_x-subalg-not-necess}
 Define $f:\mathbb Z\to\mathbb Z$ by $f(n)=n+1$. Then $\mathbb Z$ acts on $(\mathbb Z,f)$ by shifts. Consider the restriction of this action to the relative subalgebra on $\{1,2\}$. Since $D_{-1}=\{1\}$ and $f(1)=2\not\in D_{-1}$, then $\cD_{-1}$ is not a subalgebra. 
\end{exm}

\subsection{Total algebras}\label{pA(GAl(T))}

Denote by $\Al T\subseteq\pAl T$ the subcategory of total algebras of type $T$.
\begin{rem}\label{D_x-subalg-necess}
 Let $(\0,\cA)\in\pA(G,\pAl T)$, such that $\cA\in\Al T$. Then there exists a globalization of $(\0,\cA)$ in $\A(G,\pAl T)$ if and only if $\cD_x$ is a subalgebra of $\cA$ for all $x$.
\end{rem}
\noindent Indeed, in this situation \cref{glob-pA(GpAl(T))-cond} takes the form
\begin{align}\label{0_x-isom-between-subalg}
 xa_1,\dots,xa_{n_\g}\ne\emp\impl f_\g(xa_1,\dots,xa_{n_\g})=xf_\g(a_1,\dots,a_{n_\g})
\end{align}
(for $n_\g=0$ this simply means that $xf_\g=f_\g$). Since $\cA$ is total, $f_\g(xa_1,\dots,xa_{n_\g})\ne\emp$, provided that $xa_1,\dots,xa_{n_\g}\ne\emp$, so
$$
 xa_1,\dots,xa_{n_\g}\ne\emp\impl xf_\g(a_1,\dots,a_{n_\g})\ne\emp,
$$
i.\,e. $\cD_{x\m}$ is a subalgebra.\\

The full subcategory of $\pA(G,\pAl T)$ formed by the pairs $(\0,\cA)$ with $\cA\in\Al T$ and $\cD_x$ being a subalgebra of $\cA$ will be denoted by $\pA(G,\Al T)$. It contains the subcategory of global actions $\A(G,\Al T)=\pA(G,\Al T)\cap\A(G,\pAl T)$.

\begin{prop}\label{refl-A(GAl(T))-pA(GAl(T))}
 The subcategory $\A(G,\Al T)$ is reflective in $\pA(G,\Al T)$. 
\end{prop}
\begin{proof}
 It follows from \cref{ex-glob-pA(GpAl(T))<=>R_A^U-funct} that $\A(G,\pAl T)$ is reflective in $\pA(G,\Al T)$. Considering $G$ as the category with one object, observe that $\A(G,\Al T)$ and $\A(G,\pAl T)$ are isomorphic to the categories of functors from $G$ to $\Al T$ and $\pAl T$, respectively. Since $\Al T$ is reflective in $\pAl T$ by~\cite[p. 182, Corollary 2]{Gratzer}, then $\A(G,\Al T)$ is reflective in $\A(G,\pAl T)$ thanks to~\cite[Proposition IV.4.6]{Tsalenko-Shulgeyfer}. Thus, $\A(G,\Al T)$ is reflective in $\pA(G,\Al T)$.
\end{proof}

To give the precise form of an $\A(G,\Al T)$-reflector of $(\0,\cA)\in\pA(G,\Al T)$, we recall some basic constructions from  universal algebra.

Let $X$ be an arbitrary non-empty set, whose elements are called {\it letters}. We define a {\it word over $X$} as follows:  
\begin{enumerate}
 \item each $x\in X$ is identified with the word $x$ of length $l(x)=1$;\label{x-word}
 \item if $n_\g=0$, then $f_\g$ is a word of length $l(f_\g)=1$;\label{f_g-word}
 \item if $n_\g>0$ and $w_1,\dots,w_{n_\g}$ are words, then $w=f_\g(w_1,\dots,w_{n_\g})$ is a word of length  $l(w)=l(w_1)+\dots+l(w_{n_\g})+1$;\label{f_g(w_1...w_n_g)-word}
 \item each word can be obtained using \cref{x-word,f_g-word,f_g(w_1...w_n_g)-word} finitely many times.
\end{enumerate}
The set $W(X)$ of words over $X$ forms an algebra $\cW(X)$ under the operations defined in \cref{f_g-word,f_g(w_1...w_n_g)-word}. It is {\it the free algebra of type $T$ over $X$}.

Given $\cA\in\pAl T$, consider the free algebra $\cW(A)$. {\it The value} $v(w)$ of $w\in W(A)$ will be defined as follows.
\begin{enumerate}
 \item Let $a$ be a letter. Then $v(a)=a$ as an element of $A$.
 \item Let $n_\g=0$. If $(f_\g)_\cA\ne\emp$, then $v(f_\g)=(f_\g)_\cA$; otherwise $v(f_\g)=\emp$;
 \item Let $n_\g>0$ and $w=f_\g(w_1,\dots,w_{n_\g})$. If $\emp\ne v(w_i)=a_i$, $1\le i\le n_\g$, and $\emp\ne (f_\g)_\cA(a_1,\dots,a_{n_\g})=a$, then $v(w)=a$; otherwise $v(w)=\emp$.
\end{enumerate}
Thus, $v$ is a partial function $W(A)\dashrightarrow A$. It induces a congruence $\Th$ on $\cW(A)$: $(w,w')\in\Th$ if and only if either $\emp\ne v(w)=v(w')$, or 
$$
 w=f_\g(w_1,\dots,w_{n_\g}),\ w'=f_\g(w'_1,\dots,w'_{n_\g}),\ (w_i,w'_i)\in\Th,\ 1\le i\le n_\g
$$
(see~\cite[Theorem 14.1, Lemma 14.1]{Gratzer}). The quotient of $\cW(A)$ by $\Th$ is {\it the algebra absolutely freely generated by $\cA$} in the sense of~\cite[\S 28]{Gratzer}\footnote{When $T=\{2\}$, this algebra is called the free completion of $\cA$ in~\cite[IV.7.3]{Adamek-Trnkova} and the free extension of $\cA$ in~\cite[2.5.9]{Ljapin-Evseev}. See also~\cite[V.11.3]{Maltsev}.}. We denote it by $\cW(\cA)=(W(\cA),F)$. 

\begin{rem}\label{cW(cA)-refl}
 The algebra $\cW(\cA)$ is an $\Al T$-reflector of $\cA\in\pAl T$ with $a\mapsto[a]_\Th$ being the reflector morphism.
\end{rem}
\noindent This is explained by Definition~1 on p.~180 and Corollary~2 on p.~182 from~\cite{Gratzer}.\\

The following proposition clarifies the structure of $\Th$.
\begin{prop}\label{Theta-gen-by-pairs}
 The congruence $\Th$ is generated by the pairs
\begin{align}\label{gen-of-Theta}
 (f_\g(a_1,\dots,a_{n_\g}),(f_\g)_\cA(a_1,\dots,a_{n_\g})),
\end{align}
where $\g<o(T)$, $a_1,\dots,a_{n_\g}\in A$ and $(f_\g)_\cA(a_1,\dots,a_{n_\g})\ne\emp$. 
\end{prop}
\begin{proof}
 Denote by $\vTh$ the congruence generated by \cref{gen-of-Theta}. Obviously, $\vTh\subseteq\Th$, since each pair \cref{gen-of-Theta} has the same value $(f_\g)_\cA(a_1,\dots,a_{n_\g})\in A$. 
 
 For the converse inclusion we first show by induction on $l(w)$ that $(w,v(w))\in\vTh$ for any $w\in W(A)$ with $v(w)\ne\emp$. Indeed, this is clear, when $l(w)=1$. Let $w=f_\g(w_1,\dots,w_{n_\g})$ with $n_\g>0$ and $v(w)\ne\emp$. Then $v(w_i)\ne\emp$ for all $i$ and $\emp\ne(f_\g)_\cA(v(w_1),\dots,v(w_{n_\g}))=v(w)$. Note that $(w_i,v(w_i))\in\Th$. By the induction hypothesis $(w_i,v(w_i))\in\vTh$ and hence $(w,f_\g(v(w_1),\dots,v(w_{n_\g})))\in\vTh$. It remains to use \cref{gen-of-Theta} and transitivity.
 
 We now prove by induction on $l(w)+l(w')$ that $(w,w')\in\Th\impl(u,v)\in\vTh$. If $l(w)=1$ or $l(w')=1$, then $(w,w')\in\Th$ means that $\emp\ne v(w)=v(w')$. Since $(w,v(w)),(w',v(w'))\in\vTh$, by transitivity $(w,w')\in\vTh$. Now let $w=f_\g(w_1,\dots,w_{n_\g})$, $w'=f_\delta(w'_1,\dots,w'_{n_\delta})$ with $n_\g,n_\delta>0$ and $(w,w')\in\Th$. If $\emp\ne v(w)=v(w')$, then $(w,w')\in\vTh$ as above. Otherwise $n_\g=n_\delta$ and $(w_i,w'_i)\in\Th$, $1\le i\le n_\g$. By the induction hypothesis $(w_i,w'_i)\in\vTh$ for all $i$. Then clearly $(w,w')\in\vTh$.
\end{proof}

\begin{rem}\label{(bar-0cW(A))}
 Each partial action $\0$ of $G$ on a set $A$ naturally extends to a partial action $\bar\0$ of $G$ on $\cW(A)$ by the rule:
\begin{enumerate}
 \item if $w=a$ is a letter, then $\bar\0(x,w)\ne\emp\iff\0(x,a)\ne\emp$, and $\bar\0(x,w)=\0(x,a)$ in this case;
 \item if $n_\g=0$ and $w=f_\g$, then $\bar\0(x,w)=f_\g$;
 \item if $n_\g>0$ and $w=f_\g(w_1,\dots,w_{n_\g})$, then $\bar\0(x,w)\ne\emp\iff \bar\0(x,w_i)\ne\emp$ for $1\le i\le n_\g$, and $\bar\0(x,w)=f_\g(\bar\0(x,w_1),\dots,\bar\0(x,w_{n_\g}))$ in this case.
\end{enumerate} 
\end{rem}
\noindent Indeed, \cref{1a=a,x-inv-xa=a,(xy)a=x(ya)} of \cref{part-act-defn} for $\bar\0$ can be easily proved by induction on $l(w)$, and \cref{0_x-isom-between-subalg} is immediate. Note also that $\bar\0(x,w)\ne\emp\iff \0(x,a)\ne\emp$ for any letter $a$ of $w\in W(A)$.

\begin{prop}\label{(bar-0cW(A))-refl}
 The pair $(\bar\0,\cW(A))$ is a reflector of $(\0,A)\in\pA(G)$ in the category $\pA(G,\Al T)$ seen as a subcategory of $\pA(G)$, the reflector morphism is the inclusion $i:A\to W(A)$.
\end{prop}
\begin{proof}
 The fact that $i$ is a morphism $(\0,A)\to(\bar\0,W(A))$ is trivial. Let $(\vt,\cB)\in\pA(G,\Al T)$ and $\vf$ be a morphism $(\0,A)\to(\vt,B)$ in $\pA(G)$. There exists a (unique) homomorphism $\bar\vf:\cW(A)\to\cB$ which extends the map $\vf:A\to B$. It remains to check \cref{morph-in-pA-new}, which can be easily done by induction on $l(w)$.
\end{proof}

\begin{prop}\label{(bvtcW(cA))-refl}
 Given $(\0,\cA)\in\A(G,\pAl T)$, the map $\bt(x,[w]_\Th)=[\bar\0(x,w)]_\Th$ is a well-defined action of $G$ on $\cW(\cA)$. Moreover, the pair $(\bt,\cW(\cA))$ is an $\A(G,\Al T)$-reflector of $(\0,\cA)$, the reflector morphism being $a\mapsto [a]_\Th$.
\end{prop}
\begin{proof}
To prove that $\bt$ is well-defined, it suffices to note that the image of a pair \cref{gen-of-Theta} under $\bar\0(x,-)$ is again a pair of the form \cref{gen-of-Theta}. Then clearly $(\bt,\cW(\cA))\in\A(G,\Al T)$.

Since $\bt(x,[a]_\Th)=[\bar\0(x,a)]_\Th=[\0(x,a)]_\Th$, the second assertion of the proposition follows from \cref{cW(cA)-refl} together with Propositions~IV.4.1 and IV.1.13 and the observation before Proposition~IV.4.6 from~\cite{Tsalenko-Shulgeyfer}.
\end{proof}

The next remark complements \cref{refl-A(GAl(T))-pA(GAl(T))}.
\begin{rem}\label{refl-A(GAl(T))-pA(GAl(T))-precise}
 The pair $(\bth^\U,\cW(\cA^U))$, where $\bth^\U(x,[w]_\Th)=[\ol{\0^U}(x,w)]_\Th$ for $w\in W(A^U)$, is an $\A(G,\Al T)$-reflector of $(\0,\cA)\in\pA(G,\Al T)$. The reflector morphism maps $a\in A$ to the $\Th$-class $[1,a]_\Th\in W(\cA^U)$ of the one-letter word $[1,a]\in W(A^U)$. 
\end{rem}

\begin{defn}\label{glob-A(GAl(T))-defn}
 A {\it globalization of $(\0,\cA)\in\pA(G,\Al T)$ in $\A(G,\Al T)$} is a globalization of $(\0,\cA)$ in $\A(G,\pAl T)$ which belongs to $\A(G,\Al T)$. 
\end{defn}

\begin{thrm}\label{glob-A(GAl(T))}
 Each $(\0,\cA)\in\pA(G,\Al T)$ admits a globalization in $\A(G,\Al T)$. Moreover, $(\bth^\U,\cW(\cA^U))$ is a universal globalization of $(\0,\cA)$ in $\A(G,\Al T)$.
\end{thrm}
\begin{proof}
 It is enough to prove that $(\bth^\U,\cW(\cA^U))$ is a globalization of $(\0,\cA)$ in $\A(G,\Al T)$. Taking into account~\cite[Theorem 14.2]{Gratzer}, we see that the reflector morphism $(\0,\cA)\to(\bth^\U,\cW(\cA^U))$ is injective, as $[1,-]$ is injective and $\Th^\nat$ restricted to the set of one-letter words $[x,a]\in W(A^U)$ is injective. It remains to check \cref{xa-ne-emp-iff-xiota(a)-in-iota(M)}. Suppose that there are $a,b\in A$, such that $[1,b]_\Th=\bth^\U(x,[1,a]_\Th)=[x,a]_\Th$. By~\cite[Theorem 14.2]{Gratzer} one has $[1,b]=[x,a]$ and hence $\0(x,a)\ne\emp$.
\end{proof}

\begin{rem}\label{bar-0^UcW(A^U)-refl}
 Let $(\0,A)\in\pA(G)$. Then a universal globalization of $(\bar\0,\cW(A))$ in $\A(G,\Al T)$ is isomorphic to $(\ol{\0^U},\cW(A^U))$.
\end{rem}

\subsection{Algebras with identities}\label{pA(GAl(TSigma))}

Let $T=\{n_\g\}_{\g<o(T)}$ be a type of algebras and $X$ a set. An {\it identity over $X$} is a pair of words $w,w'\in W(X)$ written as an equality $w=w'$. An algebra $\cA$ {\it satisfies an identity $w=w'$}, if for any homomorphism $\vf:\cW(X)\to\cA$ one has $\vf(w)=\vf(w')$. Given a set of identities $\Sg$, the {\it variety of algebras of type $T$ determined by $\Sg$} is the class $\Al{T,\Sg}$ of algebras of type $T$ satisfying each $\sigma\in\Sg$. 

We now fix $T$ and $\Sg$ and set $\V=\Al{T,\Sg}$. The variety $\V$ can be seen as a full subcategory of $\Al T$. 

\begin{prop}\label{Al(TSigma)-refl-in-Al(T)}
 The subcategory $\V$ is reflective in $\Al T$.
\end{prop}
\begin{proof}
 Let $\cA\in\Al T$. Define $\Phi$ to be the congruence on $\cA$ generated by the pairs $(\vf(w),\vf(w'))$, where $(w,w')\in\Sg$ and $\vf:\cW(X)\to\cA$ is a homomorphism. By Birkhoff's theorem $\cA/\Phi\in \V$ (see~\cite[Theorem 26.3]{Gratzer}), and it is easy to check that $\cA/\Phi$ is a $\V$-reflector of $\cA$.
\end{proof}

Let $G$ be a group. The embedding $\V\subseteq\Al T$ naturally defines the categories $\A(G,\V)\subseteq\pA(G,\V)\subseteq\pA(G,\Al T)$. 

\begin{cor}\label{A(GV)-refl-in-pA(GV)}
 The subcategory $\A(G,\V)$ is reflective in $\pA(G,\V)$.
\end{cor}
\noindent Indeed, using \cref{Al(TSigma)-refl-in-Al(T)} and the same argument as in the proof of \cref{refl-A(GAl(T))-pA(GAl(T))}, one concludes that $\A(G,\V)$ is reflective in $\A(G,\Al T)$. Moreover, $\A(G,\Al T)$ is reflective in $\pA(G,\Al T)$ by \cref{refl-A(GAl(T))-pA(GAl(T))}, therefore, $\A(G,\V)$ is reflective in $\pA(G,\Al T)$. In particular, $\A(G,\V)$ is reflective in $\pA(G,\V)$.

\begin{rem}\label{refl-A(GV)}
 The pair $(\wt{\bth^\U},\cW(\cA^U)/\Phi)$, where
 $$
  \wt{\bth^\U}(x,[[w]_\Th]_\Phi)=[\bth^\U(x,[w]_\Th)]_\Phi=[[\ol{\0^U}(x,w)]_\Th]_\Phi
 $$
 for $w\in W(A^U)$, is an $\A(G,\V)$-reflector of $(\0,\cA)\in\pA(G,\V)$. The reflector morphism is $\Phi^\nat\circ\Th^\nat\circ[1,-]$.
\end{rem}

We would like to write the reflector in a more convenient form. To this end we need some observations about congruences on abstract algebras.

The next fact is immediate.
\begin{lem}\label{vf^(-1)(Th)}
 Let $\vf:\cA\to\cB$ be a homomorphism and $\Th$ a congruence on $\cB$. Then $\vf\m(\Th)$ is a congruence on $\cA$.
\end{lem}

Given $\cA\in\Al T$ and $\rho\subseteq A^2$, by $\rho^*$ we denote the congruence on $\cA$ generated by $\rho$. It is the intersection of all the congruences on $\cA$ containing $\rho$.

\begin{lem}\label{vf(rho)^*}
 Let $\rho\subseteq A^2$ and $\vf:\cA\to\cB$ be a homomorphism. Then $\vf(\rho)^*=\vf(\rho^*)^*$.
\end{lem}
\begin{proof}
 Clearly, $\vf(\rho)^*\subseteq\vf(\rho^*)^*$, as $\rho\subseteq\rho^*$. For the converse inclusion observe that $\vf(\rho^*)\subseteq\vf(\Th)$ for any congruence on $\cA$, such that $\Th\supseteq\rho$. In particular using \cref{vf^(-1)(Th)} we may take $\Th$ of the form $\vf\m(\Th')$, where $\Th'$ is a congruence on $\cB$. Thus, $\vf(\rho^*)\subseteq\vf(\vf\m(\Th'))\subseteq\Th'$ for any $\Th'$ on $\cB$ with $\vf\m(\Th')\supseteq\rho$. It remains to note that the latter inclusion is equivalent to $\Th'\supseteq\vf(\rho)$.
\end{proof}

\begin{prop}\label{(A/P)/P^nat(sigma)^*}
 Let $\cA\in\Al T$, $\rho,\s\subseteq A^2$ and $P=\rho^*$, $\vS=\s^*$. Then $(\cA/P)/P^\nat(\s)^*\cong (\cA/\vS)/\vS^\nat(\rho)^*$.
\end{prop}
\begin{proof}
 We shall prove $(\cA/P)/P^\nat(\s)^*\cong\cA/(\rho\cup\sigma)^*$ (the fact that $(\cA/\vS)/\vS^\nat(\rho)^*\cong\cA/(\rho\cup\sigma)^*$ is symmetric). Clearly, $P=\rho^*\subseteq(\rho\cup\sigma)^*$, so by~\cite[Theorem 11.4]{Gratzer} one has $\cA/(\rho\cup\sigma)^*\cong(\cA/P)/((\rho\cup\sigma)^*/P)$, where $(\rho\cup\sigma)^*/P$ is the congruence on $\cA/P$ defined by
 $$
  (P^\nat(a),P^\nat(b))\in(\rho\cup\sigma)^*/P\iff(a,b)\in(\rho\cup\sigma)^*.
 $$
 Observe that $(\rho\cup\sigma)^*/P$ is exactly $P^\nat((\rho\cup\sigma)^*)$, and since the latter is a congruence, by \cref{vf(rho)^*} it coincides with $P^\nat(\rho\cup\sigma)^*$. But $P^\nat(\rho\cup\sigma)^*=P^\nat(\sigma)^*$, as $P^\nat(\rho)$ is the equality relation.
\end{proof}

Since $\cW(X)$ is free over $X$ in $\Al T$, then its reflector $\cW(X)/\Phi$ in $\V$, denoted by $\cW_\V(X)=(W_\V(X),F)$, will be free over $X$ in $\V$. It is the so-called {\it free algebra over $X$ in the variety $\V$}. We shall need the injectivity of the map $X\to W_\V(X)$, $x\mapsto[x]_\Phi$. To this end we assume that $\V$ contains nontrivial algebras (see~\cite[p. 163, Corollary 2]{Gratzer}).

Let $\cA\in\pAl T$. Applying $\Phi^\nat$ to the generators \cref{gen-of-Theta} of the congruence $\Th$ on $\cW(A)$, we obtain the pairs
\begin{align}\label{gen-of-Th_Sg}
 ((f_\g)_{\cW_\V(A)}(a_1,\dots,a_{n_\g}),(f_\g)_\cA(a_1,\dots,a_{n_\g}))\in W_\V(A)^2,
\end{align}
where $(f_\g)_\cA(a_1,\dots,a_{n_\g})\ne\emp$, and the elements of $A$ are identified with their images in $W_\V(A)$. The congruence on $\cW_\V(A)$ generated by \cref{gen-of-Th_Sg} will be denoted by $\Th_\V$ and the corresponding quotient algebra by $\cW_\V(\cA)=(W_\V(\cA),F)$.

\begin{cor}\label{W(A)/Phi_Sg-cong-W_Sg(A)}
 The algebra $\cW(\cA)/\Phi$ is isomorphic to $\cW_\V(\cA)$.
\end{cor}
\noindent Indeed, each homomorphism $\vf:\cW(X)\to\cW(\cA)$ has the form $\Th^\nat\circ\phi$, for some homomorphism $\phi:\cW(X)\to\cW(A)$. Therefore, the generators of $\Phi$ on $\cW(\cA)$ are the images under $\Th^\nat$ of the generators of the analogous congruence on $\cW(A)$. It remains to apply \cref{(A/P)/P^nat(sigma)^*}.

\begin{cor}\label{refl-A(GV)-equiv-form}
 The pair $(\bth^\U_\V,\cW_\V(\cA^U))$, where 
 $$
  \bth^\U_\V(x,[[w]_\Phi]_{\Th_\V})=[[\ol{\0^U}(x,w)]_\Phi]_{\Th_\V}
 $$
 for $w\in W(A^U)$, is an $\A(G,\V)$-reflector of $(\0,\cA)\in\pA(G,\V)$. The reflector morphism is $\Th_\V^\nat\circ\Phi^\nat\circ[1,-]$.
\end{cor}
\noindent This follows from \cref{refl-A(GV),W(A)/Phi_Sg-cong-W_Sg(A)}.

\begin{defn}\label{glob-A(GV)}
 A {\it globalization of $(\0,\cA)\in\pA(G,\V)$ in $\A(G,\V)$} is a globalization of $(\0,\cA)$ in $\A(G,\Al T)$ which belongs to $\A(G,\V)$.
\end{defn}

\begin{thrm}\label{glob-pA(GV)}
 There exists a globalization of $(\0,\cA)\in\pA(G,\V)$ in $\A(G,\V)$ if and only if
 \begin{align}\label{([xa][yb])-in-Th_Sg=>[xa]=[yb]}
  ([x,a],[y,b])\in\Th_\V\impl[x,a]=[y,b]
 \end{align}
 for all $[x,a],[y,b]\in A^U$ seen as the elements of $W_\V(A^U)$. In this case the reflector $(\bth^\U_\V,\cW_\V(\cA^U))$ is a universal globalization of $(\0,\cA)$ in $\A(G,\V)$.
\end{thrm}
\begin{proof}
 If $(\0,\cA)$ is globalizable, then the reflector of $(\0,\cA)$ is its globalization. Suppose that
 \begin{align}\label{([xa][yb])-in-Th_Sg}
 ([x,a],[y,b])\in\Th_\V.
 \end{align}
 Then $y\m x[1,a]_{\Th_\V}=y\m[x,a]_{\Th_\V}=[1,b]_{\Th_\V}$. Hence $(y\m x)a\ne\emp$ by \cref{xa-ne-emp-iff-xiota(a)-in-iota(M)} and thus $[x,a]=[y,(y\m x)a]$. It follows from \cref{([xa][yb])-in-Th_Sg} that $[1,(y\m x)a]_{\Th_\V}=[1,b]_{\Th_\V}$. Therefore $(y\m x)a=b$, since the reflector morphism is injective in this case. So, $[x,a]=[y,b]$.
 
 Conversely, assume \cref{([xa][yb])-in-Th_Sg=>[xa]=[yb]}. We shall prove that the reflector is a globalization. It immediately follows that the reflector morphism is injective. Let us verify \cref{xa-ne-emp-iff-xiota(a)-in-iota(M)}. If $[1,b]_{\Th_\V}=x[1,a]_{\Th_\V}=[x,a]_{\Th_\V}$, then $[1,b]=[x,a]$, so $\emp\ne xa=b$.
\end{proof}

The following example is inspired by~\cite[Example 3.5]{DE}.

\begin{exm}\label{non-inj-refl-morph}
 Let $\sem$ be the variety of semigroups. Consider $\cS=(S,\cdot)\in\sem$, where $S=\{0,u,v,t\}$ with $0$ being zero and 
 $$
 u^2=v^2=t^2=uv=vu=ut=tu=0,\ \ vt=tv=u.
 $$ 
Take $G=\la x\mid x^2=1\ra$ and define $\0:G\times S\dashrightarrow S$ by $x0=0$, $xu=v$, $xv=u$, $xt=\emp$, $\0(1,-)=\id_S$. Then $(\0,\cS)\in\pA(G,\sem)$. Observe that in $W_\sem(S^U)$
\begin{align*}
 [1,v]&=[x,u]\sim_{\Th_\sem}[x,v][x,t]=[1,u][x,t]\sim_{\Th_\sem}[1,t][1,v][x,t]\\
 &=[1,t][x,u][x,t]\sim_{\Th_\sem}[1,t][x,0]=[1,t][1,0]\sim_{\Th_\sem}[1,0].
\end{align*}
Consequently, $(\0,\cS)$ is not globalizable in $\A(G,\sem)$ by \cref{glob-pA(GV)}.
\end{exm}

\section{Partial actions and generalized amalgams}\label{amalg}

Let $\V=\Al{T,\Sg}$ be a fixed variety of algebras. The following definitions should be well-known,  but we could not find an appropriate reference.

\begin{defn}\label{amalg-defn}
By a {\it (generalized) amalgam of $\V$-algebras} we mean a triple
\begin{align*}
 \bA=[\{\cA_i\}_{i\in I},\{\cA_{ij}\}_{i,j\in I}, \{\alpha_{ij}\}_{i,j\in I}],
\end{align*}
where $\cA_i\in\V$, $\cA_{ij}$ is a subalgebra of $\cA_i$ and $\alpha_{ij}$ is an isomorphism $\cA_{ij}\to\cA_{ji}$ with $\alpha_{ji}=\alpha_{ij}\m$, $i,j\in I$ (for convenience we shall assume that $A_{ii}=A_i$ and $\alpha_{ii}=\id$). 
\end{defn}

 
\begin{defn}\label{embed-defn}
We say that an amalgam $\bA$ is {\it embeddable} into an algebra $\cA\in\V$, if there exist injective morphisms $\vf_i:\cA_i\to\cA$, such that $\vf_j\circ\alpha_{ij}=\vf_i|_{A_{ij}}$ and $\vf_i(A_i)\cap\vf_j(A_j)=\vf_i(A_{ij})$ (which coincides with $\vf_j(A_{ji})$).
\end{defn}

\begin{defn}\label{amalg-free-prod-defn}
 Given an amalgam $\bA$, the quotient of the free algebra $\cW_\V(\bigsqcup A_i)$ by the congruence $N$ generated by
 \begin{align}
    &\bigcup_{k,\g}\{((f_\g)_{\cW_\V(\bigsqcup A_i)}(a_1,\dots,a_{n_\g}),(f_\g)_{\cA_k}(a_1,\dots,a_{n_\g}))\mid a_1,\dots,a_{n_\g}\in A_k\}\label{gen-mult-in-free-prod}\\
  &\cup\{(a,\alpha_{ij}(a))\mid i,j\in I, a\in A_{ij}\}\label{gen-amalg-prod},
 \end{align}
 will be denoted by $\prod^*_{\cA_{ij}=\cA_{ji}}\cA_i$ and called the {\it free product of algebras $\cA_i$ with amalgamated subalgebras $\cA_{ij}$}.
\end{defn}

The maps $\nu_i(a)=[a]_N$, $a\in A_i$, are homomorphisms $\cA_i\to\prod^*_{\cA_{ij}=\cA_{ji}}\cA_i$, as $N$ contains \cref{gen-mult-in-free-prod}, and they satisfy $\nu_j\circ\alpha_{ij}=\nu_i|_{A_{ij}}$, since $N$ contains \cref{gen-amalg-prod}. Moreover, $\prod^*_{\cA_{ij}=\cA_{ji}}\cA_i$ is universal among the algebras with such properties, as the next lemma shows.

\begin{lem}\label{univ-amalg-prod}
 Let $\bA$ be an amalgam, $\cA\in \V$ and $\vf_i:\cA_i\to\cA$ homomorphisms satisfying $\vf_j\circ\alpha_{ij}=\vf_i|_{A_{ij}}$. Then there is a unique homomorphism $\psi:\prod^*_{\cA_{ij}=\cA_{ji}}\cA_i\to\cA$, such that $\psi\circ\nu_i=\vf_i$.
\end{lem}
\begin{proof}
 The map $\phi:\bigsqcup A_i\to A$, $\phi|_{A_i}=\vf_i$, uniquely extends to a homomorphism $\bar\phi:\cW_\V(\bigsqcup A_i)\to\cA$. Clearly, the generators \cref{gen-mult-in-free-prod} belong to $\ker\bar\phi$, as $\vf_i$ are homomorphisms. Since
 $\bar\phi\circ\alpha_{ij}=\vf_j\circ\alpha_{ij}=\vf_i|_{A_{ij}}=\bar\phi|_{A_{ij}}$, the generators \cref{gen-amalg-prod} also belong to $\ker\bar\phi$. Hence, there exists $\psi:\prod^*_{\cA_{ij}=\cA_{ji}}\cA_i\to\cA$ satisfying $\psi\circ N^\nat=\bar\phi$, and it is uniquely defined on the generators $\nu_i(a)$ of $\prod^*_{\cA_{ij}=\cA_{ji}}\cA_i$ by $\psi\circ\nu_i(a)=\vf_i(a)$.
\end{proof}

\begin{lem}\label{emb-into-A<=>emb-into-amalg-prod}
 An amalgam $\bA$ is embeddable into a $\V$-algebra if and only if it is embeddable into $\prod^*_{\cA_{ij}=\cA_{ji}}\cA_i$.
\end{lem}
\begin{proof}
 The ``if'' part is trivial. For the ``only if'' part suppose that there are morphisms $\vf_i:\cA_i\to\cA$, which determine an embedding of $\bA$ into $\cA$. By \cref{univ-amalg-prod} there exists $\psi:\prod^*_{\cA_{ij}=\cA_{ji}}\cA_i\to\cA$, such that $\psi\circ\nu_i=\vf_i$. Since $\vf_i$ is injective, it follows that $\nu_i$ is also injective. 
 
 Obviously, $\nu_i(A_{ij})=\nu_j(A_{ji})\subseteq\nu_i(A_i)\cap\nu_j(A_j)$. Now if $a=\nu_i(b)=\nu_j(c)\in\nu_i(A_i)\cap\nu_j(A_j)$ for $b\in A_i$ and $c\in A_j$, then $\psi(a)=\vf_i(b)=\vf_j(c)\in\vf_i(A_i)\cap\vf_j(A_j)=\vf_i(A_{ij})$. It follows from the injectivity of $\vf_i$ that $b\in A_{ij}$. Hence, $a\in\nu_i(A_{ij})$, and thus $\nu_i(A_i)\cap\nu_j(A_j)=\nu_i(A_{ij})$.
\end{proof}

\begin{prop}\label{A-emb-in-terms-of-N}
 An amalgam $\bA$ is embeddable into a $\V$-algebra if and only if
 \begin{align}\label{(ab)-in-N=>a-in-A_(ij)-and-alpha_(ij)(a)=b}
  (a,b)\in N\impl a\in A_{ij}\ \&\ \alpha_{ij}(a)=b
 \end{align}
 for all $a\in A_i$ and $b\in A_j$ seen as elements of $W_\V(\bigsqcup A_i)$.
\end{prop}
\begin{proof}
 By \cref{emb-into-A<=>emb-into-amalg-prod} the amalgam $\bA$ is embeddable if and only if $\nu_i:\cA_i\to\prod^*_{\cA_{ij}=\cA_{ji}}\cA_i$ is injective and $\nu_i(A_i)\cap\nu_j(A_j)=\nu_i(A_{ij})$. One immediately sees that the injectivity of $\nu_i$ is equivalent to \cref{(ab)-in-N=>a-in-A_(ij)-and-alpha_(ij)(a)=b} with $i=j$.
 
 Now, assuming injectivity of $\nu_i$ for all $i$, we show that for $i\ne j$ the implication \cref{(ab)-in-N=>a-in-A_(ij)-and-alpha_(ij)(a)=b} is equivalent to the inclusion $\nu_i(A_i)\cap\nu_j(A_j)\subseteq\nu_i(A_{ij})$ (the converse inclusion is always true). Let $\nu_i(A_i)\cap\nu_j(A_j)\subseteq\nu_i(A_{ij})$. If $a\in A_i$, $b\in A_j$ and $(a,b)\in N$, then $\nu_i(a)=\nu_j(b)$. It follows that $\nu_i(a)\in\nu_i(A_{ij})$, and hence $a\in A_{ij}$ in view of the injectivity of $\nu_i$. Now $\nu_i(a)=\nu_j(\alpha_{ij}(a))$, so we conclude from the injectivity of $\nu_j$ that $\alpha_{ij}(a)=b$. Conversely, suppose that \cref{(ab)-in-N=>a-in-A_(ij)-and-alpha_(ij)(a)=b} holds. For any $c\in\nu_i(A_i)\cap\nu_j(A_j)$ there are $a\in A_i$ and $b\in A_j$ such that $c=\nu_i(a)=\nu_j(b)$. Then $(a,b)\in N$, so $a\in A_{ij}$ by \cref{(ab)-in-N=>a-in-A_(ij)-and-alpha_(ij)(a)=b}. Consequently, $c\in\nu_i(A_{ij})$.
\end{proof}

With any $(\0,\cA)\in\pA(G,\V)$ we shall associate the amalgam 
$$
 \bA(\0,\cA)=[\{\cA_x\}_{x\in G},\{\cA_{x,y}\}_{x,y\in G},\{\alpha_{x,y}\}_{x,y\in G}],
$$
where $\cA_x$ is a copy of $\cA$, $\cA_{x,y}=\cD_{x\m y}$ and $\alpha_{x,y}=\0_{y\m x}$ for all $x,y\in G$.

\begin{prop}\label{free-amalg-prod-isom-refl}
 Given $(\0,\cA)\in\pA(G,\V)$, the algebra $\cW_\V(\cA^U)$ is isomorphic to $\prod^*_{\cA_{x,y}=\cA_{y,x}}\cA_x$.
\end{prop}
\begin{proof}
 For any $a\in A$ denote by $a_x$ the copy of $a$ in $A_x$. Given $[x,a]\in A^U$, set $\psi([x,a])=[a_x]_N$. Observe that if $[x,a]=[y,b]$, i.\,e. $(y\m x)a=b$, then $a_x\in A_{x,y}$ and $\alpha_{x,y}(a_x)=b_y$. So, 
 $\psi([x,a])=\psi([y,b])$. Thus, $\psi$ is a well-defined map $A^U\to W_\V(\bigsqcup A_x)/N$, and hence it extends to a homomorphism $\cW_\V(A^U)\to\prod^*_{\cA_{x,y}=\cA_{y,x}}\cA_x$, which we denote by the same letter. Moreover,
\begin{align*}
 \psi((f_\g)_{\cW_\V(A^U)}([x,a_1],\dots,[x,a_{n_\g}]))&=[(f_\g)_{\cW_\V(\bigsqcup A_x)}((a_1)_x,\dots,(a_{n_\g})_x)]_N\\
 &=[(f_\g)_{\cA_x}((a_1)_x,\dots,(a_{n_\g})_x)]_N\\
 &=[(f_\g)_\cA(a_1,\dots,a_{n_\g})_x]_N\\
 &=\psi([x,(f_\g)_\cA(a_1,\dots,a_{n_\g})])\\
 &=\psi((f_\g)_{\cA^U}([x,a_1],\dots,[x,a_{n_\g}])).
\end{align*}
Therefore, $\psi$ induces a homomorphism $\bar\psi:\cW_\V(\cA^U)\to\prod^*_{\cA_{x,y}=\cA_{y,x}}\cA_x$, such that $\psi=\bar\psi\circ\Th_\V^\nat$. Note that $\bar\psi$ maps a generator $[x,a]_{\Th_\V}$ of $\cW_\V(\cA^U)$ to the generator $[a_x]_N$ of $\prod^*_{\cA_{x,y}=\cA_{y,x}}\cA_x$.

We now define $\eta:\cW_\V(\bigsqcup A_x)\to W_\V(\cA^U)$ by $\eta(a_x)=[x,a]_{\Th_\V}$. Observe that
\begin{align*}
 \eta((f_\g)_{\cW_\V(\bigsqcup A_x)}((a_1)_x,\dots,(a_{n_\g})_x))&=[(f_\g)_{\cW_\V(A^U)}([x,a_1],\dots,[x,a_{n_\g}])]_{\Th_\V}\\
 &=[(f_\g)_{\cA^U}([x,a_1],\dots,[x,a_{n_\g}])]_{\Th_\V}\\
 &=[x,(f_\g)_\cA(a_1,\dots,a_{n_\g})]_{\Th_\V}\\
 &=\eta((f_\g)_\cA(a_1,\dots,a_{n_\g})_x)\\
 &=\eta((f_\g)_{\cA_x}((a_1)_x,\dots,(a_{n_\g})_x)).
\end{align*}
Moreover, for $a_x\in A_{x,y}= D_{x\m y}$ one has $[x,a]=[y,(y\m x)a]$, so $\eta(a_x)=\eta(((y\m x)a)_y)=\eta(\alpha_{x,y}(a_x))$.
Thus, there exists a homomorphism $\bar\eta:\prod^*_{\cA_{x,y}=\cA_{y,x}}\cA_x\to\cW_\V(\cA^U)$ with $\eta=\bar\eta\circ N^\nat$. It maps a generator $[a_x]_N$ of $\prod^*_{\cA_{x,y}=\cA_{y,x}}\cA_x$ to the generator $[x,a]_{\Th_\V}$ of $\cW_\V(\cA^U)$.

It remains to note that the compositions $\bar\eta\circ\bar\psi$ and $\bar\psi\circ\bar\eta$ are identity on the generators.
\end{proof}

\begin{thrm}\label{refl-glob<=>amalg-emb}
 Let $(\0,\cA)\in\pA(G,\V)$. There exists a globalization of $(\0,\cA)$ in $\A(G,\V)$ if and only if $\bA(\0,\cA)$ is embeddable into a $\V$-algebra.
\end{thrm}
\begin{proof}
 By \cref{glob-pA(GV)} the pair $(\0,\cA)$ is globalizable in $\A(G,\V)$ if and only if \cref{([xa][yb])-in-Th_Sg=>[xa]=[yb]} is satisfied. Observe using the isomorphism from \cref{free-amalg-prod-isom-refl} that
 $$
  ([x,a],[y,b])\in\Th_\V\iff(a_x,b_y)\in N.
 $$
 Moreover, 
 $$
  [x,a]=[y,b]\iff a\in D_{x\m y}\ \&\ (y\m x)a=b\iff a_x\in A_{x,y}\ \&\ \alpha_{x,y}(a_x)=b_y. 
 $$
 Thus, \cref{([xa][yb])-in-Th_Sg=>[xa]=[yb]} is equivalent to \cref{(ab)-in-N=>a-in-A_(ij)-and-alpha_(ij)(a)=b}, the latter being a criterion of embeddability of $\bA(\0,\cA)$ into a $\V$-algebra by \cref{A-emb-in-terms-of-N}.
\end{proof}

\begin{cor}\label{variety-with-AP}
 Suppose that $\V$ has the amalgamation property~\cite[p. 395]{Gratzer} (for example, this holds for groups, inverse semigroups, but does not hold for semigroups). Then any partial action of a group of order 2 on a $\V$-algebra admits a globalization.
\end{cor}

\begin{cor}\label{D_x=emp}
 Let $\V$ be a variety and $\cA\in\V$ be retractable~\cite[p. 143]{Cohn}. Consider $(\0,\cA)\in\pA(G,\V)$, such that $\cD_x$ is the minimal subalgebra of $\cA$ for all $x\ne 1$. Then $(\0,\cA)$ has a globalization in $\A(G,\V)$, which is an action of $G$ on the free product $\prod^*_{x\in G}\cA_x$ of copies of $\cA$.\footnote{See also~\cite[Example 3.5]{DRS}.}
\end{cor}

In what follows $\grp$ denotes the variety of groups.

\begin{rem}\label{necess-cond-of-amalg-of-gr}
 Given $(\0,\cH)\in\pA(G,\grp)$, the amalgam $\bA(\0,\cH)$ satisfies the necessary conditions 3.0 from~\cite{HNeumann48}: if $H_{x,y,z}:=H_{x,y}\cap H_{x,z}$, then $\alpha_{x,y}(H_{x,y,z})=H_{y,z,x}$ and for any $h\in H_{x,y,z}$ one has $\alpha_{y,z}\circ\alpha_{x,y}(h)=\alpha_{x,z}(h)$.
\end{rem}
\noindent Indeed, by (ii) of~\cite[Definition 1.2]{E1} we have $y\m x( D_{x\m y}\cap D_{x\m z})= D_{y\m z}\cap D_{y\m x}$, and $z\m y(y\m x\cdot h)=z\m x\cdot h$ thanks to \cref{(xy)a=x(ya)} of \cref{part-act-defn}.

\begin{cor}\label{suff-cond-glob-grp}
 Let $(\0,\cH)\in\pA(G,\grp)$. Each one of the following conditions is sufficient for $(\0,\cH)$ to have a globalization in $\A(G,\grp)$:
 \begin{enumerate}
  \item $\cH$ is locally cyclic;\label{H-loc-cycl}
  \item $|G|=3$ and $\cH$ is abelian.\label{|G|=3}
 \end{enumerate}
\end{cor}

\noindent Indeed, \cref{H-loc-cycl} follows from~\cite[Theorem~6.1]{HNeumann50} saying that an amalgam of locally cyclic groups is embeddable into a group (and even into an abelian group); \cref{|G|=3} is explained by the fact that an amalgam of three abelian groups is embeddable into an (abelian) group (see~\cite[9.0]{HNeumann51}).\\

The following example shows that the result of \cref{|G|=3} of \cref{suff-cond-glob-grp} cannot be extended to non-abelian $\cH$.
\begin{exm}\label{non-glob-pa-of-Z_3}
Let $G=\la x\mid x^3=1\ra$, $\cH=\la a,b\mid a\m ba=b^2\ra$ and $\0_x:\la a\ra\to\la b\ra$, $\0_x(a)=b$. Then $(\0,\cH)\in\pA(G,\grp)$ and $(\0,\cH)$ does not have a globalization in $\A(G,\grp)$. Indeed, $\bA(\0,\cH)$ is the amalgam from~\cite[p. 549]{BNeumann54}, whose free product with amalgamated subgroups is the trivial group.
\end{exm}

\begin{rem}\label{pa-of-Z_3-in-pA_InvSem}
 Note that the partial action $(\0,\cH)$ from \cref{non-glob-pa-of-Z_3} has no globalization in $\A(G,\sem)$ as well, since otherwise its amalgam $\bA(\0,\cH)$ would be embeddable in the (semigroup) amalgamating product $\prod^*_{\cH_{x,y}=\cH_{y,x}}\cH_x$, which is in fact a group, so $(\0,\cH)$ would be globalizable in $\A(G,\grp)$. In particular, this gives us an example of a non-globalizable partial action of a group of order 3 on an inverse semigroup.
\end{rem}


\section{Partial actions on semigroups whose domains are ideals}\label{D_x-ideal}

Let $\alg$ be the variety of (associative, not necessarily unital) algebras over a fixed field. In~\cite{DE} the authors considered $(\0,\cA)\in\pA(G,\alg)$ with $D_x$ being an ideal in $\cA$ for all $x\in G$. By an {\it enveloping action of $(\0,\cA)$} was meant $(\vt,\cB)\in\A(G,\alg)$ with an injective morphism $\iota:(\theta,\cA)\to(\vt,\cB)$, such that $\iota(\cA)$ is an ideal in $\cB$ and $\bigcup_{x\in G}x\iota(\cA)$ generates $\cB$. Theorem~4.5 from~\cite{DE} says that $(\0,\cA)$ with unital $\cA$ admits an enveloping action exactly when the ideals $\cD_x$ are unital algebras. 

The above result extends to partial actions on (left) $s$-unital rings (this means that $a\in\cA a$ for all $a\in\cA$). If all $\cD_x$ are $s$-unital and some extra condition on multipliers of $\cA$ holds, then a globalization exists and it is unique (\cite[Theorem 3.1]{DRS}). Observe that an $s$-unital ring $\cA$ is idempotent in the sense that $\cA^2=\cA$, so one could ask if there is an analogue of the globalizations theorems for idempotent rings. We shall give an answer to the question in the context of partial actions on semigroups.

Another class of rings, partial actions on which have been well studied, are semiprime rings. Recall that $\cA$ is semiprime, whenever $aAa=0\impl a=0$ for all $a\in A$. Any partial action of $G$ on a semiprime ring $\cA$, whose domains are nonzero ideals of $\cA$, admits a globalization, as it was proved in~\cite[Theorem 1.6]{Ferrero06}. Observe that the multiplicative semigroup of a semiprime ring is reductive, in particular, it is weakly reductive (see~\cite[pp. 9--11]{Clifford-Preston-1}).

\begin{thrm}\label{glob-A(GSem)-D_x-ideal}
Let $(\0,\cS)\in\pA(G,\sem)$, such that $D_x$ is an ideal of $S$ for all $x\in G$. Then $(\0,\cS)$ has a globalization in $\A(G,\sem)$ if and only if 
\begin{align}\label{x(x-inv(su)t)=sx(x-inv(u)t)}
 x(x\m(su)t)=sx(x\m(u)t) 
\end{align}
for all $x\in G$, $u\in D_x$ and $s,t\in S$.
\end{thrm}
\begin{proof}
The ``only if'' part is obvious: if $(\vt,\cT)$ is a globalization of $(\0,\cS)$, then \cref{x(x-inv(su)t)=sx(x-inv(u)t)} trivially holds for $(\vt,\cT)$, and hence, in particular, for $(\0,\cS)$.

For the ``if'' part observe that $\Th_\sem$ is generated by the following abstract reduction system~\cite[Definition 1.1.1]{Terese}:
$$
(\cW_\sem(S^U),\{w[x,s][x,t]w'\to w[x,st]w'\mid w,w'\in W_\sem(S^U),\ x\in G,\ s,t\in S\}).
$$
Moreover, each one-letter word $[x,s]$ is a normal form~\cite[Definition 1.1.13 (i)]{Terese}. So, in view of \cref{glob-pA(GV)} the pair $(\0,\cS)$ is globalizable if and only if $\to$ has the unique normal form property~\cite[Definition 1.1.13 (v)]{Terese}.

Since $l(w[x,st]w')<l(w[x,s][x,t]w')$, the relation $\to$ is strongly normalizing~\cite[Definition 1.1.13 (iii)]{Terese}, so by Theorem~1.2.1 and (i) of Theorem~1.2.2 from~\cite{Terese} it suffices to prove that $\to$ is weakly confluent~\cite[Definition 1.1.8 (iii)]{Terese}. It is clearly enough to consider the reductions of a word of the form $[x,s][x,u][y,t]$, where $[x,u]=[y,v]$, that is
\begin{align*}
[x,s][x,u][y,t]&\to[x,su][y,t],\\
[x,s][x,u][y,t]&\to[x,s][y,vt].
\end{align*}
Note that $u$ belongs to the ideal $D_{x\m y}$ and therefore
\begin{align}\label{[xsu][yt]}
[x,su][y,t]=[y,(y\m x)(su)][y,t]\to[y,(y\m x)(su)t].
\end{align}
Similarly it follows from $v=(y\m x)u\in D_{y\m x}$ that
\begin{align}\label{[xs][yvt]}
[x,s][y,vt]=[x,s][x,(x\m y)((y\m x)(u)t)]\to[x,s(x\m y)((y\m x)(u)t)].
\end{align}
By \cref{x(x-inv(su)t)=sx(x-inv(u)t)} the reductions \cref{[xsu][yt],[xs][yvt]} give the same word.
\end{proof}


\begin{rem}\label{GrilletPetrich}
 If $|G|=2$, then, in view of \cref{refl-glob<=>amalg-emb}, the result of \cref{glob-A(GSem)-D_x-ideal} is a particular case of~\cite[Theorem 1]{Grillet-Petrich70}.
\end{rem}

\begin{rem}\label{exm-of-DE}
 Observe that in \cref{non-inj-refl-morph} the domain $D_x=\{0,u,v\}$ is an ideal of $\cS$, but $x(x\m(tu)t)=0\ne u=tx(x\m(u)t)$.
\end{rem}


\begin{cor}\label{D_x-idemp-or-weakly-red}
 Under the conditions of \cref{glob-A(GSem)-D_x-ideal} each one of the following assumptions is sufficient for $(\0,\cS)$ to be globalizable in $\A(G,\sem)$:
 \begin{enumerate}
  \item $\cD_x$ is idempotent for all $x\in G\setminus\{1\}$;\label{D_x-idemp}
  \item $\cD_x$ is weakly reductive for all $x\in G\setminus\{1\}$.\label{D_x-weakly-red}
 \end{enumerate}
 In particular, \cref{D_x-idemp} (as well as \cref{D_x-weakly-red}) is true, when $\cS$ is inverse or each $\cD_x$ is unital.
\end{cor}
\noindent Indeed, it is enough to check \cref{x(x-inv(su)t)=sx(x-inv(u)t)} for $x\ne 1$. If $u',u''\in D_{x\m}$, then $x(x\m(su'u'')t)=su'x(x\m(u'')t)=sx(x\m(u'u'')t)$, so \cref{D_x-idemp} implies \cref{x(x-inv(su)t)=sx(x-inv(u)t)}. Assuming now \cref{D_x-weakly-red} and taking $v',v''\in D_x$, we obtain \cref{x(x-inv(su)t)=sx(x-inv(u)t)} as a consequence of $v'x(x\m(su)t)=x(x\m(v'su)t)=v'sx(x\m(u)t)$ and $x(x\m(su)t)v''=sux(tx\m(v''))=sx(x\m(u)t)v''$.\\


For partial actions whose domains are unital ideals a stronger result holds.

\begin{thrm}\label{D_x-unital}
Suppose that $\cD_x$ is unital, that is $D_x=1_xS$ for some central idempotent $1_x\in S$. Set $[x,s]*[y,t]=[x,s(x\m y)(1_{y\m x}t)]$. Then $*$ is a well-defined associative operation on $S^U$. Moreover, $(\0^U,(S^U,*))$ is a (non-universal) globalization of $(\0,\cS)$ in $\A(G,\sem)$ and $[1,S]$ is a (unital) ideal of $(S^U,*)$. 
\end{thrm}
\begin{proof}
First of all notice that
\begin{align}\label{0_x(D_(x^(-1))D_y)}
 x(1_{x\m}1_y)=1_x1_{xy}
\end{align}
for all $x,y\in G$. Indeed, this follows from (ii) of~\cite[Definition 1.2]{E1} and the fact that $D_x\cap D_y=D_xD_y=1_x1_yS$ for arbitrary $x,y\in G$.

We now prove that $*$ is a well-defined. Suppose that 
\begin{align}
[x,s]&=[z,u],\label{[xs]=[zu]}\\
[y,t]&=[w,v].\label{[yt]=[wv]}
\end{align}
We need to show that
\begin{align}\label{[xs]*[yt]=[zu]*[wv]}
[x,s(x\m y)(1_{y\m x}t)]=[z,u(z\m w)(1_{w\m z}v)].
\end{align}
It follows from \cref{[xs]=[zu]} that
\begin{align}\label{(z^(-1)x)(s(x^(-1)y)(1_(y^(-1)x)t))}
 (z\m x)(s(x\m y)(1_{y\m x}t))=u(z\m x)(1_{x\m z}(x\m y)(1_{y\m x}t)).
\end{align}
Observe using \cref{0_x(D_(x^(-1))D_y)} that
\begin{align}
 (z\m x)(1_{x\m z}(x\m y)(1_{y\m x}t))&=(z\m x)(1_{x\m z}1_{x\m y}(x\m y)(1_{y\m x}t))\notag\\
&=(z\m x)((x\m y)(1_{y\m z}1_{y\m x}t))\notag\\
&=(z\m y)(1_{y\m z}1_{y\m x}t).\label{(z^(-1)x)(1_(x^(-1)z)(x^(-1)y)(1_(y^(-1)x)t))}
\end{align}
Using \cref{[yt]=[wv],0_x(D_(x^(-1))D_y)}, we see that
\begin{align}
(z\m y)(1_{y\m z}1_{y\m x}t)=(z\m w)((w\m y)(1_{y\m z}1_{y\m x}t))=(z\m w)(1_{w\m z}1_{w\m x}v).\label{(z^(-1)y)(1_(y^(-1)z)1_(y^(-1)x)t)}
\end{align}
Thus, by \cref{[xs]=[zu],(z^(-1)x)(s(x^(-1)y)(1_(y^(-1)x)t)),(z^(-1)x)(1_(x^(-1)z)(x^(-1)y)(1_(y^(-1)x)t)),(z^(-1)y)(1_(y^(-1)z)1_(y^(-1)x)t),0_x(D_(x^(-1))D_y)}
\begin{align*}
 (z\m x)(s(x\m y)(1_{y\m x}t))&=u(z\m w)(1_{w\m z}1_{w\m x}v)\\
&=u1_{z\m x}(z\m w)(1_{w\m z}v)\\
&=u(z\m w)(1_{w\m z}v),
\end{align*}
proving \cref{[xs]*[yt]=[zu]*[wv]}.

For the associativity of $*$ take $[x,s],[y,t],[z,u]\in S^U$ and observe that
\begin{align}
([x,s]*[y,t])*[z,u]&=[x,s(x\m y)(1_{y\m x}t)]*[z,u]\notag\\
&=[x,s(x\m y)(1_{y\m x}t)(x\m z)(1_{z\m x}u)],\label{([xs]*[yt])*[zu]}\\
[x,s]*([y,t]*[z,u])&=[x,s]*[y,t(y\m z)(1_{z\m y}u)]\notag\\
&=[x,s(x\m y)(1_{y\m x}t(y\m z)(1_{z\m y}u))].\label{[xs]*([yt]*[zu])}
\end{align}
To prove that \cref{([xs]*[yt])*[zu]} coincides with \cref{[xs]*([yt]*[zu])}, it suffices to show that
$$
1_{x\m y}(x\m z)(1_{z\m x}u)=(x\m y)(1_{y\m x}(y\m z)(1_{z\m y}u)),
$$
the latter being explained by \cref{(z^(-1)x)(1_(x^(-1)z)(x^(-1)y)(1_(y^(-1)x)t))}. 

Clearly, $[1,s]*[1,t]=[1,st]$, so the map $[1,-]$ is a monomorphism $\cS\to(S^U,*)$. Moreover, $[1,s]*[x,t]=[1,sx(1_{x\m}t)]$ and $[x,t]*[1,s]=[x,tx\m(1_xs)]=[1,x(tx\m(1_xs))]$, so $[1,S]$ is an ideal of $(S^U,*)$.

We finally prove that $(\0^U,(S^U,*))$ is a globalization of $(\0,\cS)$ in $\A(G,\sem)$. It is enough to show that $\0^U$ respects $*$. For $[x,s],[y,t]\in S^U$ and $z\in G$ one has
\begin{align*}
z([x,s]*[y,t])&=z[x,s(x\m y)(1_{y\m x}t)]=[zx,s(x\m y)(1_{y\m x}t)]\\
&=[zx,s(x\m z\m\cdot zy)(1_{(zy)\m\cdot zx}t)]\\
&=[zx,s]*[zy,t]=(z[x,s])*(z[y,t]).
\end{align*}
\end{proof}

\begin{rem}\label{S^U-inverse}
 Under the conditions of \cref{D_x-unital} if $\cS$ is inverse, then $(S^U,*)$ is inverse.
\end{rem}
\noindent Indeed, observe that $[x,s]*[x,t]=[x,st]$. Therefore, $[x,s]*[x,s\m]*[x,s]=[x,s]$, so $(S^U,*)$ is regular. Moreover, each idempotent of $(S^U,*)$ is of the form $[x,e]$, where $e\in E(S)$, and hence the idempotents commute.
%
%
%

\bibliography{bibl-pact}{}
\bibliographystyle{acm}

\end{document}